%% file: a16.tex
\documentclass{amsart}
\usepackage{xspace,amssymb}
\usepackage[small]{diagrams}
\input{command}\theorems\arrowsDStandard
\begin{document}
\input{title.tex}
\input{intro.tex}
\input{s01.tex}

\bibliography{../tex/papers}
\bibliographystyle{../tex/hamsplain}

\end{document}

%% file: command.tex
\def\mat#1{\ensuremath{#1}\xspace}
\def\makedef#1#2{\expandafter\gdef\csname #1\endcsname{#2}}
\def\makelet#1#2{\expandafter\let\csname #1\expandafter\endcsname\csname #2\endcsname}
\def\defmath#1{\makelet{temp@#1}{#1}\makedef{#1}{\mat{\csname temp@#1\endcsname}}}

\def\defbb#1{\makedef{c#1}{\mat{\mathbb{#1}}}}		
\def\defcal#1{\makedef{l#1}{\mat{\mathcal{#1}}}}	
\def\bbcal#1{\defbb{#1}\defcal{#1}}
\bbcal A
\bbcal B
\bbcal C
\bbcal D
\bbcal E
\bbcal F
\bbcal G
\bbcal H
\bbcal I
\bbcal J
\bbcal K
\bbcal L
\bbcal M
\bbcal N
\bbcal O
\bbcal P
\bbcal Q
\bbcal R
\bbcal S
\bbcal T
\bbcal U
\bbcal V
\bbcal W
\bbcal X
\bbcal Y
\bbcal Z

\def\al{\mat{\alpha}}
\def\be{\mat{\beta}}
\def\ga{\mat{\gamma}}

\def\de{\mat{\delta}}

\def\hi{\mat{\chi}}

\def\la{\mat{\lambda}}

\def\Om{\mat{\Omega}}

\def\te{\mat{\theta}}

\defmath{eta}
\defmath{xi}
\defmath{mu}
\defmath{nu}
\defmath{Phi}
\defmath{Psi}
\defmath{pi}
\defmath{rho}

\def\mrm@#1{\mat{\mathrm{#1}}}

\def\deffrak#1{\makedef{g#1}{\mat{\mathfrak{#1}}}} 
\deffrak b
\deffrak c
\deffrak C
\deffrak d
\deffrak D
\deffrak h
\deffrak j
\deffrak k

\deffrak m
\deffrak n
\deffrak N
\deffrak p
\deffrak P
\deffrak q
\deffrak r
\deffrak s
\deffrak t
\deffrak u
\deffrak v

\def\DMO{\DeclareMathOperator}

\DMO{\Hom}{Hom}
\DMO{\RHom}{RHom}
\DMO{\lHom}{\lH\mathit{om}}
\DMO{\Ext}{Ext}
\DMO{\lExt}{\lE\mathit{xt}}
\DMO{\End}{End}
\DMO{\Aut}{Aut}
\DMO{\Fun}{Fun}
\DMO{\Tor}{Tor}
\DMO{\ext}{ext}

\DMO{\Ob}{Ob}
\DMO{\Mor}{Mor}
\DMO{\im}{im}
\DMO{\coim}{coim}
\DMO{\coker}{coker}
\DMO{\Arr}{Arr}

\DMO{\Id}{Id}
\DMO{\id}{id}
\DMO{\add}{add} 
\DMO{\ind}{ind} 
\DMO{\pro}{pro} 
\DMO{\Map}{Map} %
\DMO{\Iso}{Iso} %
\DMO{\Isom}{Isom}%
\DMO{\Ind}{Ind}
\DMO{\Bun}{Bun}
\DMO{\Higgs}{Higgs}
\DMO{\Hitch}{Hitch}

\DMO{\Loc}{Loc}
\DMO{\Maps}{Maps}
\makeatletter
\@ifpackageloaded{bbm}
{}
{}
\makeatother

\DMO{\Presh}{Presh}

\DMO\coalg{Coalg}
\DMO{\Rep}{Rep}
\DMO{\Cor}{Cor}

\DMO{\Mod}{Mod}
\DMO{\rad}{rad}
\DMO{\soc}{soc}
\DMO{\ann}{ann}
\DMO{\pd}{pd}

\DMO{\Spec}{Spec}
\DMO{\Specm}{Specm}
\DMO{\Max}{Max}
\DMO{\spec}{Spec}
\DMO{\Proj}{Proj}
\DMO{\supp}{supp}
\DMO{\Coh}{Coh}
\DMO{\coh}{coh}
\DMO{\Qcoh}{QCoh}
\DMO{\QCoh}{QCoh}
\DMO{\Pic}{Pic}
\DMO{\Div}{Div}
\DMO{\ch}{ch}
\DMO{\Hilb}{Hilb}
\DMO{\Fitt}{Fitt}
\DMO{\Quot}{Quot}

\DMO{\Gras}{Gr}
\DMO{\Grass}{Gr}
\DMO{\Flag}{Flag}
\DMO{\Jac}{Jac}

\DMO{\cone}{cone}
\DMO{\Tw}{Tw}
\DMO{\rank}{rk}
\DMO{\rk}{rk}
\DMO{\codim}{codim}
\DMO{\cov}{cov}
\DMO{\sgn}{sgn}
\DMO{\td}{td}
\DMO{\GL}{GL}
\DMO{\PGL}{PGL}
\DMO{\SL}{SL}
\DMO{\SO}{SO}

\DMO\Der{Der}
\DMO\der{Der}
\DMO\coder{Coder}
\DMO{\diag}{diag}
\DMO{\HMod}{HMod} 
\DMO{\ad}{ad}
\DMO{\Ad}{Ad}
\DMO*{\colim}{colim}
\DMO*{\hocolim}{hocolim}
\DMO*{\holim}{holim}
\DMO{\Ho}{Ho}

\DMO{\har}{char}
\DMO{\sk}{sk}
\DMO{\cosk}{cosk}
\DMO{\Gal}{Gal}
\DMO{\tr}{tr}
\DMO{\Tr}{Tr}
\DMO{\Sh}{Sh}
\DMO{\Is}{Is} 
\DMO{\Hol}{Hol} 
\DMO{\Lie}{Lie} 
\DMO{\Res}{Res} 
\DMO{\irr}{irr} %
\DMO{\Irr}{Irr} %
\DMO{\Exp}{Exp} %
\DMO{\Log}{Log} %
\DMO{\Pow}{Pow}
\DMO{\pow}{pow}
\DMO{\mult}{mult} %
\DMO{\height}{ht} %
\DMO{\wt}{wt}
\DMO{\Vect}{Vect}
\DMO{\moda}{mod}
\DMO{\hd}{hd} 
\DMO{\face}{face}
\DMO{\Sym}{Sym}
\DMO{\Com}{Com} 
\DMO{\Eig}{Eig} 
\DMO{\cl}{cl} 
\DMO{\Li}{Li} 
\DMO{\Imxx}{Im} 
\DMO{\Rexx}{Re}

\def\n#1{\mat{\lvert#1\rvert}}

\def\iso{\simeq}

\def\tl#1{\mat{\tilde{#1}}}

\def\what#1{\mat{\widehat{#1}}}

\def\sb{\subset}

\def\xx{\times}

\def\ms{\backslash} 
\def\pser#1{[\![#1]\!]} 

\def\half{\mat{\frac12}}
\def\oh{\half} 

\def\inv{^{-1}}

\def\ang#1{\mat{\left\langle #1\right\rangle}}

\def\set#1{\mat{\{ #1\}}}
\def\sets#1#2{\mat{\{ #1 \mid #2\}}}

\def\emb{\hookrightarrow}
\def\mto{\mapsto}
\def\arr{\futurelet\test\arrtest}
\def\arrtest{\ifx^\test\let\next\arra\else\let\next\arrb\fi\next}
\def\arra^#1{\xrightarrow{#1}} \def\arrb{\to}

\def\arrowsD{
\def\mto{{\:\vrule height .9ex depth -.2ex width .04em\!\!\!\;\ar}}
\def\ar{{\:\vrule depth -.52ex height .60ex width 0.85em\;\!\!\rhla\,}}
\def\arr{\futurelet\test\arrtest}
\def\arrtest{\ifx^\test\let\next\arra\else\let\next\arrb\fi\next}
\def\arra^##1{\rTo^{##1}} \def\arrb{\ar}
\def\emb{\futurelet\test\embtest}
\def\embtest{\ifx^\test\let\next\emba\else\let\next\embb\fi\next}
\def\emba^##1{\rInto^{##1}} \def\embb{{\:\rthooka\!\!\!\ar}}
\newarrow{Eq}=====
\def\rrarr{\pile{\rTo\\ \rTo}}
\def\lrarr{\pile{\rTo\\ \lTo}}   
\newarrow{ShortTo}{}{}-->
}

\def\arrowsDStandard{
\newarrow{TeXto}----{->}
\newarrow{TeXinto}C---{->}
\newarrow{TeXonto}----{->>}
\newarrow{TeXdashto}{}{dash}{}{dash}{->}
\newarrow{Eq}=====
\def\ar{\rightarrow}
\def\emb{\futurelet\test\embtest}
\def\embtest{\ifx^\test\let\next\emba\else\let\next\embb\fi\next}
\def\emba^##1{\rTeXinto^{##1}} \def\embb{\hookrightarrow}
\def\rrarr{\pile{\rTo\\ \rTo}}
\def\lrarr{\pile{\rTo\\ \lTo}}   
}

\def\theorems{
\newcounter{nthr} 
\numberwithin{nthr}{section}
\let\theHnthr\thenthr
\newtheorem{thr}[nthr]{Theorem}
\newtheorem{prp}[nthr]{Proposition}
\newtheorem{lmm}[nthr]{Lemma}
\newtheorem{crl}[nthr]{Corollary}
\newtheorem{clm}[nthr]{Claim}
\newtheorem{conj}{Conjecture}
\newtheorem{rmr}[nthr]{Remark}
\theoremstyle{definition}
\newtheorem{dfn}[nthr]{Definition}
\newtheorem{exm}[nthr]{Example}
\newtheorem{claim}[nthr]{Claim}
\theoremstyle{remark}
}

\newif\ifrem\remtrue

\def\lb#1{\mat{\underline{#1}}} 
\def\ub#1{\mat{\overline{#1}}}  
\def\over#1#2{\mat{\substack{#1\\#2}}} 
\def\lqq{\lq\lq}
\def\rqq{\rq\rq\xspace}

\def\ie{i.e.\ }
\def\eg{e.g.\ }
\def\cf{cf.\ }

\def\eprint#1#2{%
\expandafter\ifx\csname eprint@#1\endcsname\relax#1:#2%
\else\def\itemID{#2}\csname eprint@#1\endcsname\fi}

\def\defArchive#1#2{%
\makedef{eprint@#1}{#2}}

\defArchive{arxiv}{\href{http://arXiv.org/abs/\itemID}{{\tt arXiv:\itemID}}}
\defArchive{numdam}{\href{http://www.numdam.org/item?id=\itemID}{Numdam:\itemID}}

%% file: title.tex
\title[Motivic DT invariants]{Motivic Donaldson-Thomas invariants and Kac conjecture}
\author{Sergey Mozgovoy}%
\email{mozgovoy@maths.ox.ac.uk}%
\begin{abstract}
We derive some combinatorial consequences from the positivity of Donaldson-Thomas invariants for symmetric quivers conjectured by Kontsevich and Soibelman and proved recently by Efimov.
These results are used to prove the Kac conjecture for quivers having at least one loop at every vertex.
\end{abstract}
\maketitle
\tableofcontents

%% file: intro.tex
\section{Introduction}
The goal of this paper is to relate two quite different topics, the motivic Donaldson-Thomas invariants and the Kac positivity conjecture. Motivic Donaldson-Thomas invariants where introduced by Kontsevich and Soibelman \cite{kontsevich_stability} for ind-constructible $3$-Calabi-Yau categories endowed with some additional data. Most easily this machinery works for quivers with potentials (see also \cite{kontsevich_cohomological}). In this paper we will work with a toy example -- a quiver with a trivial potential. While the existence of a meaningful integration map from the Hall algebra of an ind-constructible $3$-Calabi-Yau category to the quantum torus which would be an algebra homomorphism is rather difficult to prove \cite{kontsevich_stability}, the existence of such integration map from the Hall algebra of a hereditary category (\eg a category of representations of a quiver) to the quantum torus is relatively easy and well-known \cite{reineke_counting}. For symmetric quivers, this allows us to define the Donaldson-Thomas invariants in a very explicit way. It follows from the conjecture of Kontsevich and Soibelman \cite[Conjecture 1]{kontsevich_cohomological} on the properties of the cohomological Hall algebra of a symmetric quiver, that the Donaldson-Thomas invariants are polynomials with non-negative coefficients. 
An interesting combinatorial interpretation of the Donaldson-Thomas invariants for the quiver with one vertex and several loops was given by Reineke \cite{reineke_degenerate}.
The full conjecture \cite[Conjecture 1] {kontsevich_cohomological} was recently proved by Efimov \cite{efimov_cohomological}.

Let us remind the Kac conjecture now. It was shown by Kac \cite{kac_root} that for any quiver $Q$ and for any dimension vector $\al\in\cN^{Q_0}$, there exists a polynomial $a_\al(q)$ with integer coefficients such that the number of absolutely indecomposable representations of $Q$ (\ie representations that remain indecomposable after any field extension) of dimension \al over a finite field $\cF_q$ equals $a_\al(q)$. It was conjectured by Kac that the polynomials $a_\al(q)$ have non-negative integer coefficients. This conjecture was proved in \cite{crawley-boevey_absolutely} for indivisible dimension vectors \al (\ie when the greatest common divisor of the coordinates of \al is $1$).

In this paper we will prove the Kac positivity conjecture for quivers with enough loops (\ie having at least one loop at every vertex) using the positivity of the Donaldson-Thomas invariants for symmetric quivers. Our proof is based on a thorough analysis of the Hua formula \cite{hua_counting} which allows an explicit computation of the polynomials $a_\al(q)$. We will consider the refinement of the Hua formula and show that the functions arising from this refinement are polynomials with non-negative coefficients (for quivers with enough loops). The refined positivity statement implies then the Kac conjecture. Interestingly enough, the refined positivity statement is not true for quivers that don't have enough loops (see Remark \ref{not enough}), although the Kac conjecture, as we believe, is. To prove the refined positivity statement we use the positivity of the Donaldson-Thomas invariants for symmetric quivers.

The paper is organized as follows. In Section \ref{prelim} we remind the definitions of the Hall algebra, quantum torus, integration map between them, and plethystic operations on the quantum torus. In Section \ref{DT} we define Donaldson-Thomas invariants for symmetric quivers and formulate some positivity conjectures related to them. In Section \ref{some pos} we formulate some positivity conjectures of combinatorial nature and prove that one of them is true if the conjecture on the positivity of DT invariants is true. In Section \ref{kac pos sec} we prove the Kac positivity conjecture for quivers with enough loops using the results from the previous sections.

This paper is an extended version of my talk at the workshop \lqq Representation Theory of Quivers and Finite Dimensional Algebras\rqq held in February 2011 at MFO, Oberwolfach. I would like to thank the organizers for the invitation.
I would like to thank Tamas Hausel, Markus Reineke, and Fernando Rodriguez-Villegas for many useful discussions. I would like to thank Bernhard Keller and Yan Soibelman for useful comments.

The author's research was supported by the EPSRC grant EP/G027110/1.

%% file: s01.tex
\def\qt{\what{\mathbb T}}
\section{Preliminaries}
\label{prelim}
\subsection{Hall algebra and quantum torus}
Let $Q=(Q_0,Q_1)$ be a quiver. Let $\hi$ be the corresponding Euler-Ringel form. Let
$$\ang{\al,\be}=\hi(\al,\be)-\hi(\be,\al),\qquad \al,\be\in\cZ^{Q_0}$$
be the anti-symmetric form of $Q$ and let $T(\al)=\hi(\al,\al)$ be the Tits form of $Q$. 

Let $H$ be the Hall algebra of $Q$ over a finite field $\cF_q$ (we use the conventions from \cite{kontsevich_stability} which give an opposite of the usual Ringel-Hall algebra). Its basis as a vector space consists of all isomorphism classes of representations of $Q$ over $\cF_q$. Multiplication is given by the rule
$$[N]\circ[M]=\sum_{[X]}F_{MN}^X[X],$$
where
$$F_{MN}^X=\#\sets{U\sb X}{U\iso N,\ X/U\iso M}.$$
The algebra $H$ is graded by the dimension of representations. Let $\what H$ be its completion with respect to this grading.

We define the quantum torus $\qt=\qt_Q$ as follows. As a vector space it is $$\cQ(q^\oh)\pser{x_1,\dots,x_r},$$
where $r=\#Q_0$ ($q$ will be either a power of prime number or a new variable, depending on the context). Multiplication is given by
$$x^{\al}\circ x^{\be}=(-q^\oh)^{\ang{\al,\be}}x^{\al+\be}.$$

\begin{prp}[\cf Reineke \cite{reineke_counting}]
The map
$$I:\what H\to \qt,\qquad [M]\mto \frac{(-q^\oh)^{T(\lb\dim M)}}{\#\Aut M}x^{\lb\dim M}$$
is an algebra homomorphism.
\end{prp}

\subsection{Semistable representations}
Let $\te\in\cR^{Q_0}$. For any $\al\in\cN^{Q_0}\ms\set0$, we define
$$\mu_\te(\al)=\frac{\te\cdot\al}{\sum\al_i}.$$
For any $Q$-representation $M$, we define $\mu_\te(M)=\mu_\te(\lb\dim M)$, where $\lb\dim M=(\dim M_i)_{i\in Q_0}\in\cN^{Q_0}$ is the dimension vector of $M$. We say that a $Q$-representation $M$ is semistable (resp.\ stable) if for any $0\ne N\subsetneq M$, we have $\mu_\te(N)\le\mu_\te(M)$ (resp.\ $\mu_\te(N)<\mu_\te(M)$). For any
$\al\in\cN^{Q_0}$, we define
$$\tl A_\al^\te=\sum_{\over{M\text{ is }\te-sst}{\lb\dim M=\al}}[M]\in\what H,\qquad A_\al^\te x^\al=I(\tl A^\te_\al)\in\qt.$$
For any $\mu\in\cR$, we define
$$\tl A_\mu^\te=\sum_{\mu_\te(\al)=\mu}\tl A_\al^\te
=\sum_{M\text{ is }\te-sst}[M]\in\what H,\qquad
A_\mu^\te=I(\tl A_\mu^\te)=\sum_{\mu_\te(\al)=\mu}A_\al^\te x^\al\in\qt.$$
It was proved by Markus Reineke that $A_\al(q)$ are rational functions in the variable $q^\oh$ \cite{reineke_harder-narasimhan}.

\begin{rmr}
\label{expl A}
For $\te=0$, $\mu=0$, we denote $A_\mu^\te$ just by $A$. One can easily show
(see \eg \cite[Theorem 5.1]{mozgovoy_number}) that
\begin{equation}
A=\sum_\al\frac{(-q^\oh)^{-T(\al)}}{(q\inv)_\al}x^\al,
\label{eq:expl A}
\end{equation}
where $(q)_\al=\prod_i(q)_{\al_i}$ and $(q)_n=\prod_{k=1}^n(1-q^k)$ for $n\ge0$.
\end{rmr}

\subsection{Plethystic operations}
In this section we consider $\qt$ as an algebra endowed with the usual commutative multiplication. We consider $q^\oh$ as a new variable. 
For any function $f(q^\oh,x_1,\dots,x_r)$ in $\qt$, we define the Adams operations
\begin{equation}
\psi_n(f(q^\oh,x_1,\dots,x_r))=f(q^{\oh n},x_1^n,\dots,x_r^n),\qquad n\ge1.
\label{eq:psi}
\end{equation}

We define the plethystic exponential $\Exp:\qt^+\to 1+\qt^+$ (here $\qt^+$ is the maximal ideal of $\qt$) by the rule (see \cite{getzler_mixed} or \cite[Appendix]{mozgovoy_computational} for more details)
$$\Exp(f)=\exp\left(\sum_{n\ge1}\frac1n\psi_n(f)\right).$$
Its inverse, the plethystic logarithm $\Log:1+\qt^+\to \qt^+$, is given by
$$\Log(f)=\sum_{n\ge1}\frac{\mu(n)}n\psi_n(\log(f)),$$
where $\mu(n)$ is the M\"obius function.

\begin{rmr}
\label{expl A2}
Define the operator $T:\qt\to \qt$, $x^\al\mto (-q^\oh)^{T(\al)}x^\al$. Then we can rewrite equation \eqref{eq:expl A} as
\begin{equation}
A=\sum_\al\frac{(-q^\oh)^{-T(\al)}}{(q\inv)_\al}x^\al
=T\inv\left(\sum_\al\frac{x^\al}{(q\inv)_\al}\right)
=T\inv\Exp\left(\frac{\sum x_i}{1-q\inv}\right),
\label{eq:expl A2}
\end{equation}
where the last equation follows from the Heine formula \cite{mozgovoy_fermionic}
$$\sum_{n\ge0}\frac{x^n}{(q)_n}=\Exp\left(\frac{x}{1-q}\right).$$
\end{rmr}

\section{Donaldson-Thomas invariants}
\label{DT}
Assume that $Q$ is a symmetric quiver, \ie the anti-symmetric bilinear form $\ang{-,-}$ is zero. Then $\qt$ (with the twisted multiplication) is a commutative algebra. The following definition follows \cite[Definition 21]{kontsevich_cohomological}.

\begin{dfn}
\label{DT def}
For any $\mu\in\cR$, we define the Donaldson-Thomas invariants $\Om_\mu^\te=\sum_{\mu_\te(\al)=\mu}\Om_\al^\te x^\al\in\qt$ by the formula
$$A^\te_\mu=\Exp\left(\frac{\Om_\mu^\te}{q-1}\right).$$
For the trivial stability $\te=0$, we denote $\Om_\al^\te$ by $\Om_\al$.
\end{dfn}

\begin{rmr}
Using the above formula we can define the Donaldson-Thomas invariants $\Om_\mu^\te$ for an arbitrary quiver and a slope $\mu\in\cR$ such that $\ang{\al,\be}=0$ whenever $\mu_\te(\al)=\mu_\te(\be)=\mu$.
\end{rmr}

\begin{rmr}
The classical Donaldson-Thomas invariants 
$$\ub\Om^\te_\mu
=\sum_{\mu(\al)=\mu}\ub\Om^\te_\al x^\al\in\cQ\pser{x_i,i\in Q_0}$$
are defined by the formula
$$\lim_{q\to1}(q-1)\log A_\mu^\te=\sum_{\al}\ub\Om^\te_\al\Li_2(x^\al)
=\sum_{n\ge1}\frac1{n^2}\sum_{\al}\ub\Om^\te_\al x^{n\al},$$
where the dilogarithm function $\Li_2$ is defined by $\Li_2(x)=\sum_{n\ge1}\frac{x^n}{n^2}$. If we \lqq quantize\rqq this formula, we obtain
\begin{multline*}
A^\te_\mu
=\exp\left(\frac{1}{q-1}
\sum_{n\ge1}\frac 1n\frac{q-1}{q^n-1}\sum_\al\Om_\al^\te(q^n)x^{n\al}\right)\\
=\exp\left(\sum_{n\ge1}\frac 1n\psi_n\left(\frac1{q-1}\sum_\al\Om_\al^\te(q)x^{\al}\right)\right)
=\Exp\left(\frac1{q-1}\sum_\al\Om_\al^\te(q)x^{\al}\right).	
\end{multline*}
This coincides with Definition \ref{DT def}.
\end{rmr}

\begin{rmr}
A priori, the invariants $\Om_\al^\te$ are elements of $\cQ(q^\oh)$. 
\end{rmr}

The following statement is a consequence of \cite[Conjecture 1]{kontsevich_cohomological}

\begin{conj}
\label{KS conj}
For any $\al\in\cN^{Q_0}$, the functions $\Om_\al(-q^\oh)$ are polynomials in $q^{\pm\oh}$ with non-negative integer coefficients. 
\end{conj}

\begin{rmr}
One can show that if this conjecture is true then actually $\Om_\al(-q^\oh)\in\cN[q^\oh]$, \ie there are no negative powers of $q^\oh$ in $\Om_\al$.
\end{rmr}

\begin{rmr}
\label{on proofs}
A slightly weaker statement of Conjecture \ref{KS conj} for the quivers with one vertex and several loops was recently proved by Markus Reineke \cite{reineke_degenerate}.
The complete proof of \cite[Conjecture 1]{kontsevich_cohomological} and thus of Conjecture \ref{KS conj} was recently obtained by Efimov~\cite{efimov_cohomological}.
\end{rmr}

\begin{rmr}
Consider the quiver with one vertex and $g$ loops. The Hilbert-Poincar\'e series of its cohomological Hall algebra equals \cite[Section 2.5]{kontsevich_cohomological}
$$P(x,q^\oh)=\sum_{n\ge0}\frac{(-q^\oh)^{T(n)}}{(q)_n}x^n.$$
Comparing this with the formula in Remark \ref{expl A}, we see that
$$P(x,q^\oh)=A(q^{-\oh}).$$
By \cite[Theorem 3]{kontsevich_cohomological}
$$A(q^{-\oh})=P(x,q^\oh)
=\prod_{n\ge1}\prod_{m\in\cZ}(q^{m/2}x^n,q)_\infty^{\de(n,m)}=\Exp\left(\sum_{n\ge1}\sum_{m\in\cZ}\frac{-\de(n,m)q^{m/2}x^n}{1-q}\right)$$
for some integers $\de(n,m)$, where the q-Pochhammer symbols $(x,q)_\infty$ are defined by
$$(x,q)_\infty=\prod_{i\ge0}(1-q^ix)
=\Exp\left(-x\sum_{i\ge0}q^i\right)
=\Exp\left(\frac{-x}{1-q}\right).$$
This implies
$$\Om_n(q^\oh)=q\sum_{m\in\cZ}-\de(n,m)q^{-m/2}.$$
and therefore
$$\Om_n(-q^\oh)=q\sum_{m\in\cZ}(-1)^{m-1}\de(n,m)q^{-m/2}.$$
It is explained in \cite[Section 2.6]{kontsevich_cohomological} that 
\cite[Conjecture 1]{kontsevich_cohomological} implies
$$(-1)^{m-1}\de(n,m)\ge0$$
and therefore $\Om_n(-q^\oh)\in\cN[q^{\pm\oh}]$.
\end{rmr}

\begin{exm}
Let $Q$ be a quiver with one vertex and $g$ loops. Let $\Om^{(g)}(q^\oh)=\sum_{n\ge1}\Om^{(g)}_n(q^\oh)x^n$ be the corresponding Donaldson-Thomas invariants. Then
\begin{align*}
\Om^{(0)}(-q^\oh)&=q^{\oh}x.\\
\Om^{(1)}(-q^\oh)&=qx.\\
\Om^{(2)}(-q^\oh)&=q^{3/2} x+q^3x^2+q^{11/2}x^3+(q^7+q^9)x^4+\dots.\\
\Om^{(3)}(-q^\oh)&=q^2x+q^5x^2+(q^7+q^8+q^{10})x^3\\
&\phantom{=}+(q^9+q^{10}+2q^{11}+q^{12}+2q^{13}+q^{14}+q^{15}+q^{17})x^4+\dots.
\end{align*}
\end{exm}

Computer tests show that, more generally, we should expect

\begin{conj}
\label{DT conj}
For any stability parameter $\te\in \cR^{Q_0}$ and for any $\al\in\cN^{Q_0}$, the functions $\Om_\al^\te(-q^\oh)$ are polynomials in $q^{\oh}$ with non-negative integer coefficients. 
\end{conj}

The fact that $\Om_\al^\te\in\cZ[q^{\pm\oh}]$ follows from \cite[Section 6.2]{kontsevich_cohomological}. Alternatively, we can argue as follows.

\begin{thr}[Mozgovoy-Reineke \cite{mozgovoy_number}]
For any quiver $Q$, we have
$$A_\mu\circ T\Exp\left(\frac{\sum_{\mu_\te(\al)=\mu}S_\al^\te(q) x^\al}{1-q}\right)=1$$
in $\qt$,
where the operator $T:\qt\to \qt$ is given by $x^\al\mto(-q^\oh)^{T(\al)}x^\al$ and $S_\al^\te\in\cZ[q]$ are polynomials counting absolutely \te-stable $Q$-representations of dimension \al over finite fields.
\end{thr}
Over a symmetric quiver this result implies
\begin{equation}
\Exp\left(\frac{\Om_\mu^\te}{1-q}\right)
=T\Exp\left(\frac{\sum_{\mu_\te(\al)=\mu}S_\al^\te x^\al}{1-q}\right).
\end{equation}
Now the fact that $\Om_\al^\te\in\cZ[q^{\pm\oh}]$ follows from

\begin{thr}[Kontsevich-Soibelman {\cite[Theorem 9]{kontsevich_cohomological}}]
Let $B$ be an $r\xx r$ symmetric integer matrix and let $T:\qt\to \qt$ be an operator defined by $x^\al\mto (-q^\oh)^{\al^tB\al}x^\al$. Assume that
$$\Exp\left(\frac{\sum b_\al(q)x^\al}{1-q}\right)=T \Exp\left(\frac{\sum a_\al(q)x^\al}{1-q}\right),$$
where $a_\al\in\cZ[q^{\pm\oh}]$. Then $b_\al\in\cZ[q^{\pm\oh}]$.
\end{thr}

\section{Combinatorial positivity conjectures}
\label{some pos}
Let $C$ be an $r\xx r$ matrix with non-negative coefficients (not necessarily symmetric).
We define the operator $\ub T:\qt\to \qt$ by $\ub T(x^\al)=q^{\al^tC\al}x^\al$.
For any $\al\in\cN^r$, we define $\al!=\prod_{i=1}^n\al_i!$. Computer tests give evidence for the following two conjectures.

\begin{conj}
Assume that
$$\exp\left(\frac{\sum b_\al(q) \frac{x^\al}{\al!}}{q-1}\right)=
\ub T\exp\left(\frac{\sum a_\al(q) \frac{x^\al}{\al!}}{q-1}\right),$$
where $a_\al\in\cN[q]$, $\al\in\cN^r$. Then $b_\al\in\cN[q]$, $\al\in\cN^r$.
\end{conj}

\begin{rmr}
This conjecture was proved by Markus Reineke for $r=1$ (private communication). It is still open for general $r$.
\end{rmr}

\begin{conj}
\label{pos conj}
Assume that
$$\Exp\left(\frac{\sum b_\al(q) x^\al}{q-1}\right)=
\ub T\Exp\left(\frac{\sum a_\al(q) x^\al}{q-1}\right),$$
where $a_\al\in\cN[q]$, $\al\in\cN^r$. Then $b_\al\in\cN[q]$, $\al\in\cN^r$.
\end{conj}

\begin{rmr}
In both conjectures it is important that the operator $\ub T$ comes from a quadratic form and not just from some function with non-negative values on $\cN^r$.
\end{rmr}

Assuming that Conjecture \ref{KS conj} is true (see Remark \ref{on proofs}), we can prove 
\begin{thr}
\label{pos conj proof}
Conjecture \ref{pos conj} is true.
\end{thr}
\begin{proof}
We can assume that there are only a finite number of $\al\in\cN^r$ such that $a_\al(q)\ne0$.
Let $a_\al(q)=\sum_{k\ge0}a_{\al,k}q^k$, $\al\in\cN^r$.
For every monomial of the form $q^kx^\al$ in the expression $\sum_{\al>0}a_\al(q) x^\al$ we introduce a new variable $y_{\al,k,l}$, $1\le l\le a_{\al,k}$. Totally, there are $N=\sum a_\al(1)$ new variables.
Let $C(\al,\be)=\al^t C\be$ and $C(\al)=C(\al,\al)$.
If we substitute $y_{\al,k,i}=q^kx^\al$ then, for any choice of $\lb n=(n_{\al,k,l})_{\al,k,l}\in\cN^N$, we have
$$\ub T\left(\prod y_{\al,k,l}^{n_{\al,k,l}}\right)
=\ub T\left(\prod (q^kx^\al)^{n_{\al,k,l}}\right)
=q^{C(\sum n_{\al,k,l}\al)}\prod y_{\al,k,l}^{n_{\al,k,l}}.
$$
Note that
$$C\left(\sum n_{\al,k,l}\al\right)=\sum n_{\al,k,l}n_{\al',k',l'}C(\al,\al')$$
and this defines a quadratic form $C'$ on $\cZ^N$ given by a matrix with non-negative coefficients (of the form $C(\al,\al')$).
Assume that
$$\Exp\left(\frac{\sum b'_{\lb n}(q) y^{\lb n}}{q-1}\right)=
\ub T_{C'}\Exp\left(\frac{\sum y_{\al,k,l}}{q-1}\right),$$
where $\ub T_{C'}$ is given by $y^{\lb n}\mto q^{C'(\lb n)}y^{\lb n}$.
If we will show that $b'_{\lb n}\in\cN[q]$ for $\lb n\in\cN^N$, then this will
imply Conjecture \ref{pos conj} for general $\sum a_\al(q)x^\al$. Therefore, we can assume in the Conjecture \ref{pos conj} that $\sum a_\al(q)x^\al=\sum_{i=1}^r x_i$.

Define an $r\xx r$ symmetric matrix $B$ by
$$B_{ij}=-C_{ij}-C_{ji}.$$
Let $Q$ be a symmetric quiver such that $B$ is its Ringel matrix, \ie
$$B_{ij}=\de_{ij}-\#\set{\text{arrows from }i\text{ to }j}.$$
The corresponding Tits form is given by
$$T(\al)=\al^t B\al=-2(\al^t C\al).$$
Therefore
$$\ub T(x^\al)=q^{\al^t C\al}x^\al=q^{-\oh T(\al)}x^\al
=(-q^\oh)^{-T(\al)}x^\al=T\inv(x^\al),$$
where the operator $T:\qt\to \qt$ was defined in Remark \ref{expl A2}.
Recall from Remark \ref{expl A2} that
$$A=T\inv \Exp\left(\frac{\sum x_i}{1-q\inv}\right).$$
Conjecture \ref{KS conj} implies that the functions $\Om_\al(q)$ defined by
$$\Exp\left(\frac{\sum\Om_\al(q)x^\al}{q-1}\right)=A
=T\inv\Exp\left(\frac{\sum x_i}{1-q\inv}\right)$$
are elements of $\cN[q^{\pm1}]$. 
This implies that the functions $b_\al(q)$ defined by
$$\Exp\left(\frac{\sum b_\al(q)x^\al}{q-1}\right)
=T\inv\Exp\left(\frac{\sum x_i}{q-1}\right)
$$
are also elements of $\cN[q^{\pm1}]$ (indeed, $b_\al(q)=q^{-\n\al}\Om_\al(q)$, where $\n\al=\sum\al_i$).
We have to show that they are actually polynomials in $q$. The right hand side of the last equation is contained in the ring
$$\cQ[q]\left[\frac{1}{q^n-1}|n\ge1\right]\pser{x_1,\dots,x_r}$$
because the operator $T\inv=\ub T$ as well as the plethystic exponential preserve this ring. Therefore
$$\Om_\al\in \cN[q^{\pm1}]\cap \cQ[q]\left[\frac{1}{q^n-1}|n\ge1\right].$$
This intersection coincides with $\cN[q]$.
\end{proof}

\section{Kac positivity conjecture}
\label{kac pos sec}
In this section we will prove Kac positivity conjecture for quivers with enough loops (\ie having at least one loop at every vertex). Our proof will rely on Theorem~\ref{pos conj proof}.

Let $Q$ be an arbitrary quiver with $r$ vertices and let $\al\in\cN^{Q_0}$. It was proved by Kac that there exists a polynomial $a_\al\in\cZ[q]$, such that the number of absolutely stable representations of $Q$ of dimension \al over a finite field $\cF_q$ equals $a_\al(q)$. Kac conjectured that $a_\al\in\cN[q]$.

There exists a rather explicit formula for the polynomials $a_\al$ due to Hua \cite{hua_counting} (see also \cite{mozgovoy_computational} for its interpretation using \la-rings). Let us remind it.
Let \lP be the set of all partitions.
Given a multipartition 
$$\la=(\la^i)_{i\in Q_0}\in\lP^{Q_0},$$
where $\la^i=(\la^i_1,\la^i_2,\dots)$,
we define $\la_k=(\la^i_k)_{i\in Q_0}\in\cN^{Q_0}$, $k\ge1$.
Define the generating function
$$r(q)
=\sum_{\la\in\lP^{Q_0}}\prod_{k\ge1}
\frac{q^{-T(\la_k)}}{(q\inv)_{\la_k-\la_{k+1}}}x^{\la_k}\in\cQ(q)\pser{x_i,i\in Q_0},$$
where $(q)_\al,\al\in\cN^{Q_0,}$ was defined in Remark \ref{expl A}.
Then Hua's formula says that
$$r(q)=\Exp\left(\frac{\sum a_\al(q)x^\al}{q-1}\right).$$

We will introduce now certain generating function in the algebra 
$$A=\cQ(q)\pser{x_{ki}|k\ge1,1\le i\le r}$$
that generalizes $r(q)$.
For any $\al\in\cN^{Q_0}$, $k\ge1$, we will denote $\prod_{i\in Q_0}x_{ki}^{\al_i}$ by $x_k^\al$. Define
$$s(x_{11},\dots,x_{1r},x_{21},\dots,x_{2r},\dots)(q)=\sum_{\la\in\lP^{Q_0}}\prod_{k\ge1}
\frac{q^{-T(\la_k)}}{(q\inv)_{\la_k-\la_{k+1}}}x_k^{\la_k-\la_{k+1}}\in A.$$
The generating function $r(q)$ is obtained from $s(q)$ by substituting
$x_{ki}=x_i^k$.
The Kac positivity conjecture for quivers with enough loops follows from

\begin{thr}
\label{refined positivity}
Assume that the quiver $Q$ has enough loops (equivalently,
the matrix of the Euler-Ringel form has only non-positive components).
Then the coefficients of 
$$(q-1)\Log(s)$$
are polynomials in $q$ with non-negative coefficients.
\end{thr}

\begin{proof}
It is enough to prove the theorem for 
$$s(x_{11},\dots,x_{nr})=s(x_{11},\dots,x_{nr},0,\dots)$$
for every $n\ge1$. If we define $\ga_k=\la_k-\la_{k+1}\in\cN^{Q_0}$, $k=1\dots n$, then we can rewrite
$$s(x_{11},\dots,x_{nr})
=\sum_{\ga\in(\cN^{Q_0})^n}q^{-\sum_{k=1}^nT(\sum_{l\ge k}\ga_l)}
\prod_{k=1}^n \frac{x_k^{\ga_k}}{(q\inv)_{\ga_k}}.$$
The power of $q$ in the last formula is some quadratic form on the vectors 
$$\ga=(\ga_{ki}|1\le k\le n, 1\le i\le r)\in\cZ^{rn}.$$
The $rn\xx rn$ matrix $C$ corresponding to this quadratic form has only non-negative components (this is where we use the condition on enough loops). Let $\ub T$ be the operator defined by the matrix $C$ as in Section \ref{some pos}. We define $x^\ga=\prod_{k,i}x_{ki}^{\ga_{ki}}$ and $(q)_\ga=\prod_{k,i}(q)_{\ga_{ki}}$.
Then
$$s(x_{11},\dots,x_{nr})
=\ub T\left(\sum_{\ga\in\cN^{rn}}\frac{x^\ga}{(q\inv)_{\ga}}\right)
=\ub T\Exp\left(\frac{\sum_{k,i}x_{ki}}{1-q\inv}\right).$$
Assume that
$$\Exp\left(\frac{\sum b_\ga(q)x^\ga}{q-1}\right)
=s(x_{11},\dots,x_{nr})
=\ub T\Exp\left(\frac{\sum_{k,i}x_{ki}}{1-q\inv}\right).$$
Then by Theorem \ref{pos conj proof} the elements $b_\ga(q)$ are in $\cN[q]$.
\end{proof}

\begin{crl}
The Kac conjecture is true for quivers having at least one loop at every vertex.
\end{crl}

\begin{rmr}
Refinements of the Hua formula where studied earlier in \cite{mozgovoy_fermionic} and \cite{villegas_refinement}. We learned the idea that $(q-1)\Log(s)$ could possibly have non-negative coefficients, in the case of a quiver with one vertex and several loops, from Fernando Rodriguez-Villegas \cite{villegas_refinement}.
\end{rmr}

\begin{rmr}
\label{not enough}
Our theorem is not true for quivers containing vertices without loops. For example,
for the quiver with one vertex and without loops, we have
$$s(x_{11})=\ub T\Exp\left(\frac{x_{11}}{1-q\inv}\right),$$
where $\ub T(x_{11}^n)=q^{-n^2}x_{11}^n$. Then
$$(q-1)\Log(s)=x_{11}-q\inv x_{11}^2+q^{-3}x_{11}^3-(q^{-6}+q^{-8})x_{11}^4+\dots$$
\end{rmr}
\enlargethispage{2\baselineskip}